\documentclass[11pt]{article}
\usepackage{amsmath,amsthm,amssymb,amscd}
\usepackage{graphicx}
\usepackage{graphics}
\usepackage{verbatim}

\usepackage{mathrsfs}
\usepackage{color}
\usepackage[top=1in, bottom=1in, left=1.25in, right=1.25in]{geometry}
\usepackage{enumerate}
\usepackage[normalem]{ulem}

\newtheorem{theorem}{Theorem}[section]
\newtheorem{proposition}[theorem]{Proposition}

\newtheorem{lemma}[theorem]{Lemma}

\newtheorem{corollary}[theorem]{Corollary}

\begin{document}

\title{Frustration-critical signed graphs}

\author{Chiara Cappello\thanks{Member of the NRW-Forschungskolleg Gestaltung von flexiblen Arbeitswelten}, Eckhard Steffen\thanks{
		Paderborn University,
		Department of Mathematics,
		Warburger Str. 100,
		33098 Paderborn,
		Germany; ccapello@mail.uni-paderborn.de (C.~Cappello), es@upb.de (E.~Steffen)}}
\date{}

\maketitle

\begin{abstract}
{
A signed graph $(G,\Sigma)$ is a graph $G$ together
with a set $\Sigma \subseteq E(G)$ of negative edges. A circuit is positive if the product of the signs of its edges is positive. A signed graph $(G,\Sigma)$ is balanced if all its circuits are positive.
The frustration index $l(G,\Sigma)$ is the minimum cardinality of a set $E \subseteq E(G)$
such that $(G-E,\Sigma-E)$ is balanced, and $(G,\Sigma)$ is $k$-critical if
$l(G,\Sigma) = k$ and $l(G-e, \Sigma - e)<k$, for every $e \in E(G)$. 

We study decomposition and subdivision of critical signed graphs
and completely determine  the set of $t$-critical signed graphs, for $t \leq 2$.
Critical signed graphs are characterized. 
We then focus on non-decomposable critical signed graphs. In particular, we characterize the set $S^*$ 
of non-decomposable $k$-critical signed graphs not containing
a decomposable $t$-critical signed subgraph for every $t \leq k$. We prove 
that $S^*$ consists of cyclically 4-edge-connected projective-planar cubic graphs. Furthermore, we construct $k$-critical signed graphs of $S^*$ for every $k \geq 1$.}
\end{abstract}

\section{Introduction} \label{Introduction}

The study of structural balance in (social) networks goes back to Heider \cite{Heider_1944}
and to Cartwright and Harary \cite{Cartwright_Harary_1956}, who set Heider's theory
into the graph theoretical concept of balance in signed graphs. 
From that on, signed graphs became a very active area of research.
The dynamic survey \cite{Bibliography_Zaslavsky} gives an impression of the vast existing literature on signed graphs and related topics.

Signed graphs find application in many different scientific disciplines. In particular,
balance in signed graphs has been studied extensively in various application contexts, see e.g.~\cite{Aref_2018}. Of particular interest is the frustration of a signed graph, 
which indicates how far a signed graph is from being balanced. 

We study signed graphs which are critical with respect to frustration.
Any unbalanced signed graph contains critical subgraphs. We give structural characterizations of 
these signed graphs and therefore information on substructures of unbalanced graphs.

\subsection{Definitions and basic results}

Let $G$ be a graph. The vertex set of $G$ is denoted by $V(G)$ and the edge set by $E(G)$. Graphs may contain multiedges or loops.
A \emph{signed graph} $(G,\Sigma)$ is a graph $G$ together with a set $\Sigma \subseteq E(G)$, which 
is the set of negative edges of $(G,\Sigma)$. The set $\Sigma$ is also called a \emph{signature} of $(G,\Sigma)$.  	
A set $B$ of edges is \emph{negative} if $|B \cap \Sigma|$ is odd. 
We say that an edge $e$ has a negative sign or simply it is negative if $\{e\}$ is negative and it has
a positive sign or simply it is positive  
otherwise.

An \emph{edge-cut} in $(G,\Sigma)$ is a set of edges 
$\partial_{(G,\Sigma)}(U) = \{ uv \in E(G) \colon u \in U, v \not \in U\}$. 
The cardinality of $\partial_{(G,\Sigma)}(U)$ is denoted by $d_{(G,\Sigma)}(U)$. If $U = \{v\}$,
then we write $\partial_{(G,\Sigma)}(v)$ and $d_{(G,\Sigma)}(v)$ for $\partial_{(G,\Sigma)}(\{v\})$ and $d_{(G,\Sigma)}(\{v\})$, respectively.   
Furthermore, we denote $d_{(G,\Sigma)}^-(v) = |\partial_{(G,\Sigma)}(U) \cap \Sigma|$ and 
$d_{(G,\Sigma)}^+(U) = d_{(G,\Sigma)}(U) - d_{(G,\Sigma)}^-(U)$.  An edge-cut $\partial_{(G,\Sigma)}(U)$ is
\emph{equilibrated} if $d_{(G,\Sigma)}^+(U) = d_{(G,\Sigma)}^-(U)$. 
If there is no ambiguity we will omit the index $(G,\Sigma)$.

A signed graph is \emph{balanced} if it contains no negative circuits. 
Otherwise, we say that $(G,\Sigma)$ is \emph{unbalanced}.

Two signatures $\Sigma$ and $\Gamma$ on $G$ are \emph{equivalent} if there 
is an edge-cut $\partial(U)$ in $(G,\Sigma)$ such that 
$\Gamma = \Sigma \Delta \partial(U)$, where $\Delta$ denotes the symmetric difference of two sets. 
Hence, two signatures on $G$
are equivalent if they have the same collection of negative circuits
\cite{Z82_signed_graphs}. If $\Sigma$ and $\Gamma$ are equivalent signatures on $G$, then we also say that $\Gamma$ is a signature of $(G,\Sigma)$ and that $\Gamma$ is obtained by switching at $U$. 
It is well known that $(G,\Sigma)$ is balanced if and only if $\emptyset$ is a signature of $(G,\Sigma)$ \cite{Harary_1954}.
If $E(G)$ is a signature of $(G,\Sigma)$, then $(G,\Sigma)$ is called
\emph{antibalanced}, and we denote it by $-G$. The signed graph obtained from a graph $H$ by replacing every edge by two edges, one negative and one positive, is denoted by $\pm H$. 

The \emph{frustration index} of $(G,\Sigma)$, denoted by $l(G,\Sigma)$,
is the minimum cardinality of a set $E \subseteq E(G)$, so that
$(G-E, \Sigma-E)$ is balanced. 
A signed graph $(G,\Sigma)$ with $l(G,\Sigma) = k$ is called  \emph{ $k$-frustrated}. A signature $\Gamma$ of $(G,\Sigma)$ with
$|\Gamma|=t$ is called a \emph{$t$-signature}. Clearly, $t \geq l(G,\Sigma)$. 
We will often use the following well-known lemma.

\begin{lemma} \label{Min_signature}
Let $(G,\Sigma)$ be a $k$-frustrated signed graph. If $\Gamma$ is a set of $k$ edges such that $(G-\Gamma, \Sigma - \Gamma)$ is balanced, then $\Gamma$ is a signature of $(G,\Sigma)$.  	 
\end{lemma}

\begin{proof}  
	Let $(G',\Sigma')=(G-\Gamma,\Sigma - \Gamma)$. Since $(G',\Sigma')$ is balanced there is an edge-cut $\partial_{(G',\Sigma')}(U)$ in $G'$ such that 
	$\partial_{(G',\Sigma')}(U)=\Sigma'$. Since $|\Gamma| = k$, it follows that
	$(E(G) - \partial_{(G,\Sigma)}(U)) \cap \Sigma = (E(G) - \partial_{(G,\Sigma)}(U)) \cap \Gamma$. 
	Thus, $e \in (\Sigma-\partial_{(G,\Sigma)}(U)) \cup (\partial_{(G,\Sigma)}(U) - \Sigma)$
	if and only if $e \in \Gamma$ and therefore,
	$\partial_{(G,\Sigma)}(U) \Delta \Sigma = \Gamma$. 
\end{proof}

Clearly, the frustration index of a signed graph is invariant under
switching. It measures how far a signed graph is from being balanced. 
A trivial upper bound for the frustration index of a loopless signed graph
$(G,\Sigma)$ is $|E(G)|/2$. This bound is sharp as the signed graph $\pm H $
shows. For $k$-signatures of $k$-frustrated signed graphs,
the following well-known statement holds.

\begin{lemma} \label{edge-cut}
	If $\Sigma$ is a $k$-signature of a $k$-frustrated signed graph $(G,\Sigma)$, then $d^-_{(G,\Sigma)}(U) \leq d^+_{(G,\Sigma)}(U)$ for every $U \subseteq V(G)$.
	Furthermore, if $G$ is $n$-edge-connected, then $G-\Sigma$
	is $\lceil \frac{n}{2} \rceil$-edge-connected.
\end{lemma}

\begin{proof}
	If $d^+_{(G,\Sigma)}(U) < d^-_{(G,\Sigma)}(U)$, 
	then $|\Sigma \Delta \partial_{(G,\Sigma)}(U)| < |\Sigma|$.
	Hence, at most half of the edges of an edge-cut are in 
	$\Sigma$ and the statements follow.	
\end{proof}

There are many applications of the frustration index of a signed graph
in pure mathematical research. Sometimes different names, as e.g. negativeness or minimum odd cycle cover,
are used for this parameter. 
If $(G,\Sigma)$ is antibalanced, then
the frustration index is the minimum cardinality of a set
of edges whose removal gives a bipartite graph. 
Thus, questions on the
frustration index of $-G$ are questions on the 
maximum size of bipartite subgraphs of $G$. In this case, the 
frustration index is also called the bipartite frustration index,
the edge traversal number etc.

We study signed graphs which are critical with 
respect to the
frustration index. Let $k \geq 1$ be an integer, a $k$-frustrated 
signed graph $(G,\Sigma)$ is \emph{ $k$-critical} if $l(G-e,\Sigma-e) < k$,
for every edge $e$. 
In Section \ref{sec:criticality} we give a characterization for criticality in signed graphs and describe some classes of these graphs.
In Section \ref{Section_decomp_subdiv} we investigate decomposition and subdivision of 
critical signed graphs. These results are used to characterize the classes of $2$-critical signed graphs (and trivially of $1$-critical signed graphs) in Section \ref{Section 1-2-critical}.

In Section \ref{Section_Non_decomp} we study non-decomposable critical signed graphs. It turns out that
the class of non-decomposable critical signed graphs which do not have a decomposable 
critical subgraph is the class of critical signed graphs which do not have two edge-disjoint
negative circuits. Examples for such graphs
are subgraphs of the Escher walls (see \cite{Reed_Mangoes1999}) and they 
can be characterized as a specific class of
projective-planar signed cubic graphs. The paper concludes with Section \ref{Section_Family}, where
we construct signed graphs of this class.

In the context
of flows on signed graphs partial results for hard conjectures
are obtained for the aforementioned classes of signed graphs \cite{CQES_2020, Lu_Lou_CQ_2018, RSS_2018, CQ_etal_2019}.

\section{Critical frustrated signed graphs}\label{sec:criticality}

Let $k \geq 1$. The graph with one vertex and $k$ loops is denoted by $kC_1$.
If $k=1$ we just write $C_1$. Clearly, the signed graph $-kC_1$ is
 $k$-critical. 
  
\begin{proposition} \label{critical_subgraphs}
	Let $k \geq 1$. A $k$-frustrated signed graph contains an 
	$m$-critical subgraph for every $m \in \{1, \dots,k\}$. 
\end{proposition}

\begin{proof}
	By Lemma \ref{Min_signature} we can suppose that $\Sigma$ 
	is a $k$-signature. Let $E \subseteq \Sigma$ and $|E|= k-m$.
	Then $l(G-E, \Sigma-E) \leq m$. Furthermore, 
	$l(G-E, \Sigma-E) < m$ implies $l(G,\Sigma) < k$, a contradiction.
	Hence, $(G-E,\Sigma-E)$ is $m$-frustrated. 
	In order to obtain an $m$-critical subgraph of $(G,\Sigma)$ 
	remove step-wise those edges whose removal does not decrease the frustration index. 
\end{proof}

Next we characterize critical signed graphs.

\begin{theorem} \label{characterizations}
	Let $k \geq 1$ be an integer and $(G,\Sigma)$ be a $k$-frustrated signed graph. The following statements are equivalent. 
	
	\begin{enumerate}
		\item $(G,\Sigma)$ is $k$-critical.
		\item Every edge is contained in a $k$-signature of $(G,\Sigma)$.
		\item If $\Gamma$ is a $k$-signature of $(G,\Sigma)$, then every positive edge is contained in an equilibrated edge-cut
		of $(G,\Gamma)$. 
	\end{enumerate}
\end{theorem}

\begin{proof}
($1 \rightarrow 2$) 
Let $e \in E(G)$, $G'=G-e$ and $\Sigma'=\Sigma - \{e\}$ (it might be that $e \notin \Sigma$). 
Then $l(G', \Sigma ') = k-1$.
Let $\Gamma$ be a set of $k-1$ edges of $G'$ such that $(G'- \Gamma, \Sigma '- \Gamma )$ is balanced.
By Lemma \ref{Min_signature}, $\Gamma$ is a signature of $(G', \Sigma ')$ and therefore, $\Gamma \cup \{e\}$ is a signature of $(G, \Sigma)$.

($2 \rightarrow 3$) Let $\Gamma$ be a $k$-signature of $(G,\Sigma)$
and $e \not \in \Gamma$. There is a $k$-signature
$\Gamma'$ of $(G,\Sigma)$ which contains $e$. Thus,
there is an edge-cut $\partial_{(G,\Gamma)}(U)$
such that $\partial_{(G,\Gamma)}(U) \Delta \Gamma = \Gamma'$.
By Lemma \ref{edge-cut}, 
$d^-_{(G,\Gamma)}(U) \leq d^+_{(G,\Gamma)}(U)$.
But if $d^-_{(G,\Gamma)}(U) < d^+_{(G,\Gamma)}(U)$,
then $\Gamma'$ is a $t$-signature with $t > k$, a contradiction. Hence, 
$\partial_{(G,\Gamma)}(U)$ is equilibrated. 

($3 \rightarrow 1$) $(G,\Sigma)$ is $k$-frustrated. Thus, it has a 
$k$-signature $\Gamma$, by Lemma \ref{Min_signature}. 
There
is a $k$-signature $\Gamma_e$ with $e \in \Gamma_e$ for every
$e \in E(G)$, since every positive edge is contained in an equilibrated edge-cut.
Thus, $l(G-e,\Sigma-e) < k$ for every $e \in E(G)$. 
\end{proof}

Let $\lambda(G)$ denote the edge-connectivity
of a graph $G$.
The following corollary is an immediate consequence of Theorem \ref{characterizations}.  

\begin{corollary} \label{edge-connectivity_critical}
	Let $k \geq 1$ and $G \not = C_1$. If $(G,\Sigma)$ is a $k$-critical signed graph, then $2 \leq \lambda(G) \leq 2k$.  
\end{corollary}

We will give some examples of critical signed graphs.

\begin{proposition} \label{planar}
	Let $G$ be a plane triangulation. If $G$ has $n \geq 3$ vertices, 
	then $-G$ is $(n-2)$-critical. Furthermore, $-G$ has three 
	pairwise disjoint $(n-2)$-signatures. 
\end{proposition}

\begin{proof}
	$-G$ has $2n-4$ triangles and $3n-6$ edges. Since every edge is in at most two triangles, $l(-G) \geq n-2$. Let $G^*$ be the plane dual of $G$. $G^*$ is cubic and bridgeless. Hence, it is 3-edge colorable by the 4-Color Theorem.
	Each color class $C^*_i$ contains $n-2$ edges, and $G^*-C^*_i$ is Eulerian.
	Let $C_i$ be the set of edges of $G$ which corresponds to $C^*_i$ in $G^*$. Thus,
	$G-C_i$ is bipartite. By Lemma \ref{Min_signature}, $C_i$
	is a signature and thus, $l(-G) = n-2$. The Kempe chains
	in $G^*$, which are induced by two color classes $C^*_{i+1}$ and $C^*_{i+2}$ (indices in $\mathbb{Z}_3$) correspond to equilibrated edge-cuts in $-G$, with signature $C_i$. Thus, $-G$ has three pairwise disjoint $(n-2)$-signatures. Therefore, $-G$ is $(n-2)$-critical. 
\end{proof}

An \emph{odd wheel} $W_{2k+1}$ ($k \geq 1$) is the graph which consists of an odd circuit $C_{2k+1}$ and a vertex $v$ which is connected to 
every vertex of $C_{2k+1}$. 

\begin{proposition}
	The antibalanced odd wheel $-W_{2k+1}$ is $(k+1)$-critical.
\end{proposition}

\begin{proof}
	Let $v_0,\dots,v_{2k}$ be the vertices of $C_{2k+1}$ in this order.
	For $i \in \mathbb{Z}_{2k+1}$, if $k$ is even, then let $V_i = \{v_i, v_{i+2}, \dots, v_{i+k}, v_{i+k+1}, v_{i+k+3},\dots, v_{i-2}\}$,
	and if $k$ is odd, then
	let $V_i = \{v_{i+1}, v_{i+3}, \dots, v_{i+k}, v_{i+k+1}, v_{i+k+3},\dots, v_{i-1}\}$. 
	For each edge $e \in E(W_{2k+1})$ there exists $i$ such that $e \in \partial (V_i)$.
	Furthermore
	$|\partial(V_i)| = 3k+1$ and therefore, $E(G) -  \partial(V_i)$
	is a $(k+1)$-signature of $W_{2k+1}$. Every edge of $W_{2k+1}$ is in precisely two
	elements of $\{T^1, \dots, T^{2k+1}, C_{2k+1}\}$, where
	$T^1, \dots, T^{2k+1}$ are the $2k+1$ triangles of $W_{2k+1}$.
	It follows that
	$l(-W_{2k+1})=k+1$. Thus, $W_{2k+1}$ is $(k+1)$-critical. 
\end{proof}

The \emph{projective cube} $H_k$ of dimension $k \geq 1$ can be constructed as follows: Each vertex $v$ is labeled with a $(0,1)$-string $s(v)$ of length $k$ and $s(v) \not = s(w)$ if $v \not = w$. Two vertices are adjacent if the Hamming distance of their labels is $1$ or $k$. 

The \emph{signed projective cube} of dimension $k$ is the signed graph $(H_k, \Sigma)$ with negative edges connecting vertices whose labels
have Hamming distance $k$. 

Signed projective cubes play an exceptional role in 
the study of signed graph homomorphism, see \cite{Homomorphisms}. We will show that they are critical. 

\begin{proposition} \label{Aug_hypercube_critical}
	Let $k \geq 1$ be an integer. The signed projective cube $(H_k,\Sigma)$ of dimension $k$ is $2^{k-1}$-critical. Furthermore, it has $k+1$ pairwise disjoint $2^{k-1}$-signatures.  	
\end{proposition}

\begin{proof}
	The set $\Sigma$ is a perfect matching of $H_k$ and $l(H_k,\Sigma) \leq |\Sigma|=2^{k-1}$. The $j^{th}$ digit in $s(v)$
	is denoted by $s_j(v)$. For $i \in \{1,\dots,k\}$ let $U_i = \{v \colon s_i(v)=0\}$. Then 
	$d_{H_k}(U_i)=2^{k}$ and $\Sigma \subseteq \partial_{H_k}(U_i)$. 
	Hence, $\partial_{H_k}(U_i) \Delta \Sigma = B_i$ is a 
	$2^{k-1}$-signature and $B_1, \dots,B_k, \Sigma$ are $k+1$ pairwise disjoint $2^{k-1}$-signatures of $(H_k,\Sigma)$. Thus, $l(H_k,\Sigma) = 2^{k-1}$ and
	$(H_k,\Sigma)$ is $2^{k-1}$-critical, by Theorem \ref{characterizations}. 	
\end{proof}

It is not hard to see that the antibalanced complete
graph of order $n$ is $\lfloor \frac{(n-1)^2}{4} \rfloor$-critical. Since $K_n$ has $\frac{n(n-1)}{2}$ edges,
it follows that it cannot have three pairwise disjoint signatures 
if $n \geq 5$.

\section{Decomposition and subdivision of critical signed graphs} \label{Section_decomp_subdiv}

Let $n \geq 2$. A $k$-critical signed graph $(G,\Sigma)$ is 
\emph{ $(k_1, \dots,k_n)$-decomposable}
if it contains pairwise edge-disjoint $k_i$-critical 
subgraphs $(H_i,\Sigma_i)$ and
$k = k_1 + \dots + k_n$. If the numbers $k, k_1, \dots,k_n$  are irrelevant we just
say that $(G,\Sigma)$ is decomposable.

\begin{proposition} \label{decomp_trivial}
	Let $(G,\Sigma)$ be a $k$-critical signed graph. 
	\begin{enumerate}
	 \item If $(G,\Sigma)$ contains a negative loop or two vertices which are connected by a positive and a negative edge, then $(G,\Sigma)$
	is $(1, k-1)$-decomposable.
	
	\item If $G$ is connected and $(G,\Sigma)$ is $(k_1,k_2)$-decomposable into $(H_1,\Sigma_1)$ and $(H_2,\Sigma_2)$, then $V(H_1) \cap V(H_2) \not = \emptyset$. Furthermore, if 
	$v \in V(H_1) \cap V(H_2)$, then $d_G(v) \geq 4$. 
	
	\item If $(G,\Sigma)$ is $(k_1, \dots,k_n)$-decomposable into
	$(H_1,\Sigma_1), \dots, (H_n,\Sigma_n)$, then $\bigcup_{i=1}^n E(H_i)=E(G)$.
\end{enumerate}
\end{proposition}
\begin{proof}
	1. Let $e^+$ and $e^-$ be the two edges which are incident to the same vertices or let $e^l$ be a negative loop. By Lemma \ref{Min_signature} there is a $k$-signature $\Gamma$. Clearly, $e^l \in \Gamma$ and
	precisely one of $e^+, e^-$ is in $\Gamma$. 
	Since every equilibrated edge-cut in $(G,\Sigma)$
	gives an equilibrated edge-cut in $(G-\{e^+,e^-,e^l\}, \Sigma-\{e^+,e^-,e^l\})$, the statement follows by Theorem
	\ref{characterizations}. 
	
	2. If $V(G_1)$ and $V(G_2)$ are disjoint, then there is an edge 
	$e \in E(G) - (E(H_1) \cup E(H_2))$. Hence, $l(G-e, \Sigma-e) =
	l(G,\Sigma)$, a contradiction. By Corollary \ref{edge-connectivity_critical}, every vertex of $H_i$ has degree at least 2. Statement 3.~is proved similarly to statement 2.   
\end{proof}

By Proposition \ref{decomp_trivial}, $k$-critical
signed cubic graphs are non-decomposable. For example, see the 3-critical signed Petersen graphs $(P,\Sigma_1)$ and $(P,\Sigma_2)$ in Figure \ref{P}. 

\begin{figure}
	\centering
	\includegraphics[height=5cm]{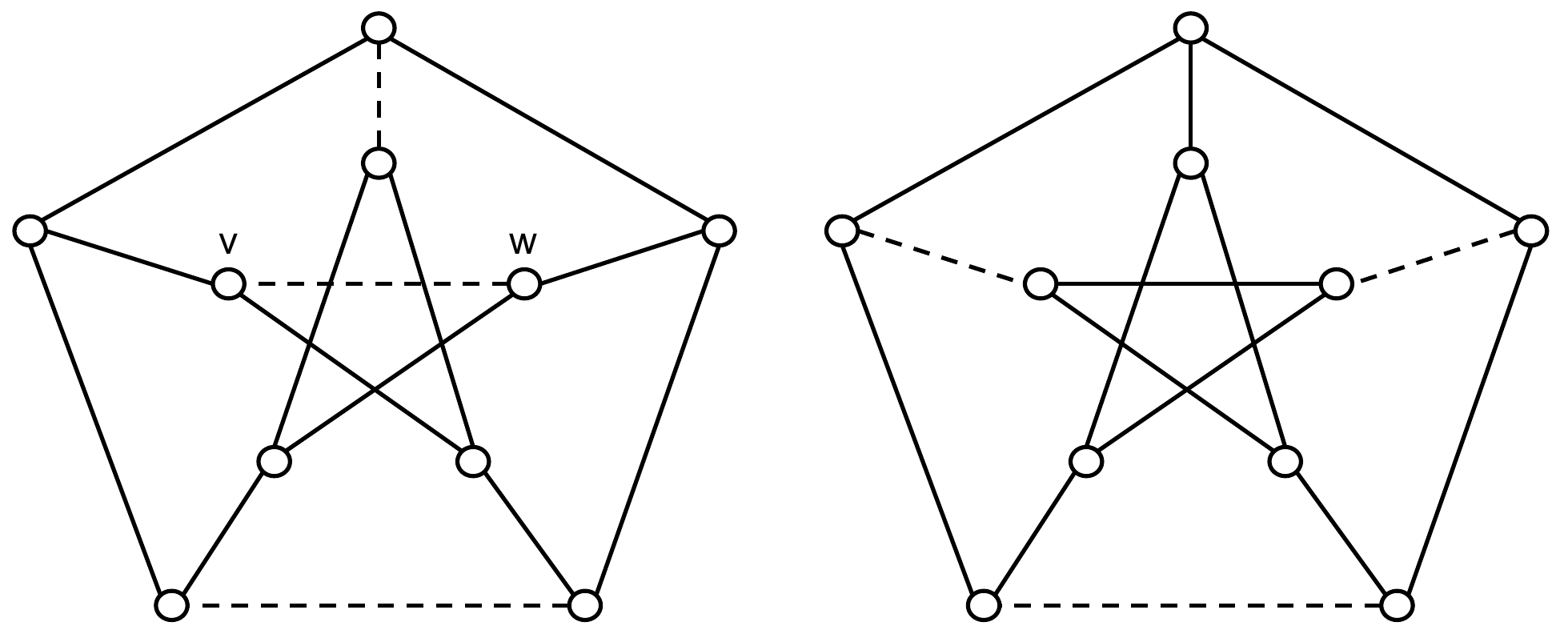}
	\caption{The dotted lines represent $\Sigma_1$ on the left and $\Sigma_2$ on the right.\label{P}} \end{figure}

Let $(G,\Sigma)$ be a signed graph and $t \geq 1$ be an integer.
A \emph{ $t$-multiedge} between two vertices $v,w$ is a set of $t$ edges
between $v$ and $w$ and it is denoted by $E_{vw}$.
A $t$-multiedge has a sign if all edges of $E_{vw}$
have the 
same sign. That is, it is positive/negative if all edges
of $E_{vw}$ are positive/negative.  
 
Let $(G,\Sigma)$ be a signed graph and $E_{xy}$ be a $t$-multiedge having a sign. 
Let $(G',\Sigma')$ be obtained from 
$(G-E_{xy}, \Sigma-E_{xy})$ by adding
a vertex $v$ and a positive $t$-multiedge $E^+_{vx}$ 
and a $t$-multiedge $E^{\Sigma}_{vy}$ which has the same sign
as $E_{xy}$. We say that $(G',\Sigma')$ is obtained from
$(G,\Sigma)$ by \emph{subdividing} a multiedge.
Furthermore, $(H',\Gamma')$ is a \emph{subdivision} of $(H,\Gamma)$ if 
$(H',\Gamma')=(H,\Gamma)$ or $(H',\Gamma')$
is obtained from $(H,\Gamma)$ by a sequence of multiedge subdivisions.
If $(H',\Gamma')$ is a subdivision of $(H,\Gamma)$ and $(H',\Gamma') \not =(H,\Gamma)$, then $(H',\Gamma')$ is called a \emph{proper subdivision} of $(H,\Gamma)$ (see Fig. \ref{fig:subdivision}) . 
 
\begin{figure}[h]
	\centering
	\includegraphics[width=0.8\linewidth]{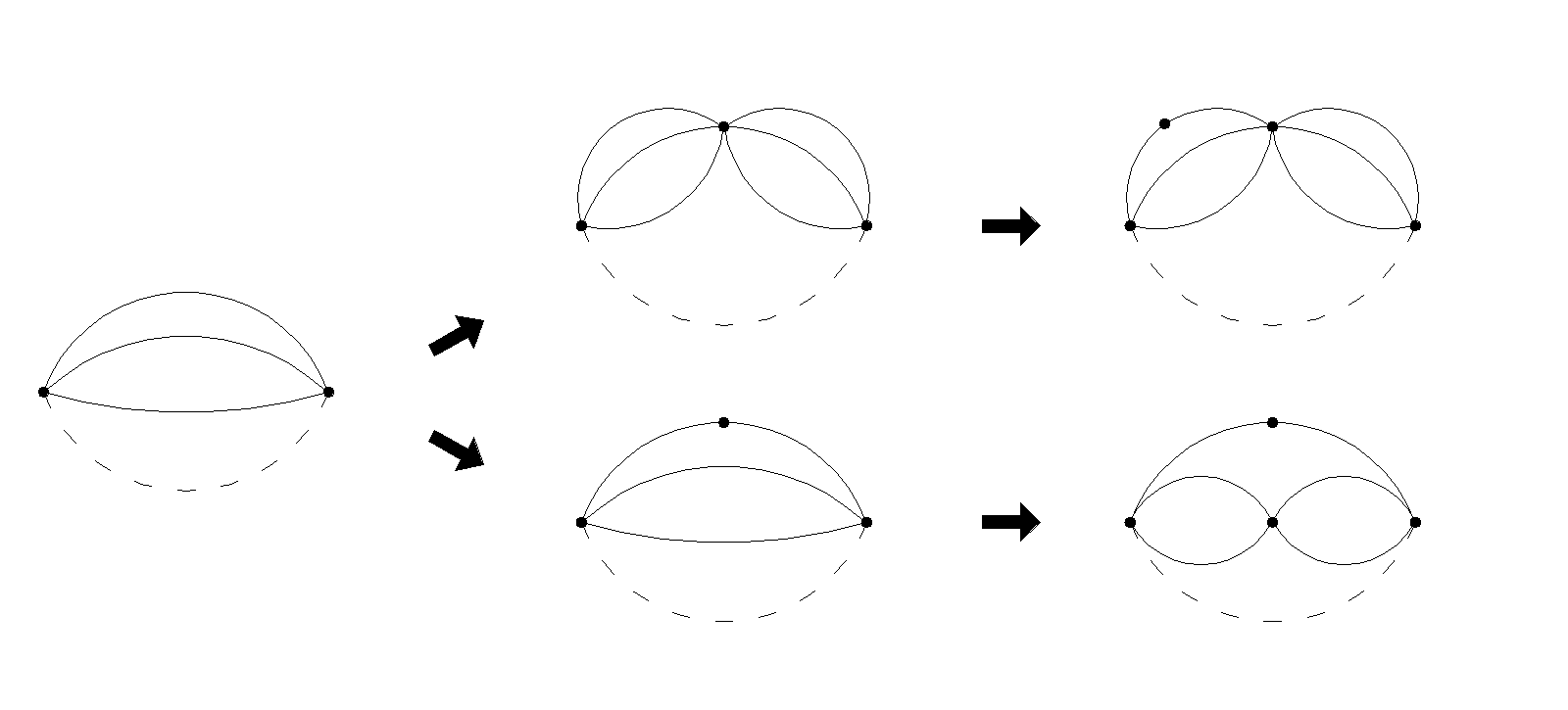}
	\caption[subdivision]{An example of how the graph depends on the order of the subdivisions.}
	\label{fig:subdivision}
	\end{figure} 

\begin{theorem} \label{Sub_Dec}
	Let $k \geq 1$ and let $(H, \Gamma)$, $(G, \Sigma)$ be two signed graphs such that $(H,\Gamma)$ is a subdivision of $(G,\Sigma)$. Then the following statements hold:
	\begin{enumerate}
		\item $l(G,\Sigma) = l(H,\Gamma)$.
		\item $(G,\Sigma)$ is $k$-critical if and only if 
			$(H,\Gamma)$ is $k$-critical. 
		\item $(G,\Sigma)$ is decomposable if and only if
		$(H,\Gamma)$ is decomposable. 
	\end{enumerate}
\end{theorem}

\begin{proof} Let $(H,\Gamma)$ be obtained from $(G,\Sigma)$ by 
	subdividing a $t$-multiedge $E_{xy}$ ($t \geq 1$), and let $v$ be the only vertex of $V(H) - V(G)$
	which subdivides $E_{xy}$. Since $G - E_{xy} = H-v$, we 
	consider an edge $e$ of $G - E_{xy}$ also as an edge of $H-v$.
	Furthermore, let $l(G,\Sigma)=k$ and $\Sigma$ be a $k$-signature
	and 
	$\Gamma$ be the signature on $H$ which is obtained from
	$\Sigma$ by the subdivision of the multiedge $E_{xy}$.  
	
	1. By definition, $l(H, \Gamma) \leq k$. Suppose to the contrary that $l(H,\Gamma) < k$. Then there is an edge-cut $\partial_{(H,\Gamma)}(U)$
	with $d^+_{(H,\Gamma)}(U) < d^-_{(H,\Gamma)}(U)$. If $x,y \in U$ or $\in V(G)-U$, then
	$\partial_{(H,\Gamma)}(U) = \partial_{(G,\Sigma)}(U)$, a contradiction. Hence, $x \in U$ and $y \not \in U$. Both cases whether $E_{xy}$ is negative or not, are covered when $v \in U$. But then $\partial_{(G,\Sigma)}(U-v)$ is an edge-cut in $(G,\Sigma)$
	with more negative than positive edges, a contradiction. Hence, 
	$l(H,\Gamma)=l(G,\Sigma)$.
	
	2. ($\rightarrow$) Let $(G,\Sigma)$ be $k$-critical. 
	By Theorem \ref{characterizations} there is a $k$-signature 
	$\Sigma_1$ such that $E_{xy} \subseteq \Sigma_1$. By construction,
	$\Gamma_1 = (\Sigma_1 - E_{xy}) \cup E^{\Sigma_1}_{vy}$ is 
	a $k$-signature of $(H,\Gamma)$ and 
	$E^+_{vx} \subset \Gamma_1 \Delta \partial_{H}(v)$, which
	is also a $k$-signature of $(H,\Gamma)$. 
	
	Every edge $e \not \in E^+_{vx} \cup E^{\Sigma_1}_{vy}$ 
	can be considered as an edge of $G - E_{xy}$. If $e$ is 
	positive, then there is an equilibrated edge-cut 
$\partial_{(G,\Sigma)}(U)$
	containing $e$. If $x \in U$ and $y \not \in U$, then
	$\partial_{(H,\Gamma)}(U')$
	with $U' = U \cup \{v\}$ is the corresponding equilibrated edge-cut in $(H,\Gamma)$. If both $x,y$ are in $U$ or
	not in $U$, then there is nothing to prove. 
	Hence, every positive edge of $(H,\Gamma)$ is contained in an
	equilibrated edge-cut and the statement follows by Theorem 
	\ref{characterizations}. 
	
	The other direction of this statement is proved similarly. 
	
	3. Note, by the definition, a decomposable signed graph is
	$k$-critical. Therefore, by 2., we assume that 
	$(G,\Sigma)$ and $(H,\Gamma)$ are $k$-critical. 
	Let $E_{xy}$, $E_{xv}$ and $E_{yv}$ be $t$-multiedges in the respective signed graphs.  
	
	($\leftarrow$) Let $(H,\Gamma)$ be $(k_1,k_2)$-decomposable
	into $(H_1,\Gamma_1)$ and $(H_2,\Gamma_2)$, $k_1+k_2=k$.	
	For $i \in \{1,2\}$ let $E^i_{xv} = E_{xv} \cap E(H_i)$
	and $E^i_{yv} = E_{yv} \cap E(H_i)$. All four
	of these multiedges have a sign in their respective graphs
	and $|E^i_{xv}|=|E^i_{yv}| = t_i$ with $t_1+t_2=t$ and $t_i \geq 0$. 
	
	Let $E^i_{xy}$ be a set of
	$t_i$ edges of $E_{x,y}$, so that $E^1_{xy} \cap E^2_{xy} = \emptyset$.
	
	For $i \in \{1,2\}$ let $(G_i, \Sigma_i)$ be the subgraph of $(G,\Sigma)$
	with $E(G_i)= E(H_i - v) \cup E^i_{xy}$ and $\Sigma_i = \Sigma \cap E(G_i)$. Then every edge of 
	$G$ is contained in precisely one of $G_1, G_2$. 
	Furthermore, $(H_i,\Gamma_i)$ is a subdivision of $(G_i,\Sigma_i)$.
	Thus, by 2., $(G_i,\Sigma_i)$ is $k_i$-critical and 
	 $(G,\Sigma)$ is $(k_1,k_2)$-decomposable into $(G_1,\Sigma_1)$ and $(G_2,\Sigma_2)$. The opposite direction
	is proved similarly starting with a decomposition of $(G,\Sigma)$.
\end{proof}	

We ask a non-empty  graph to have a non-empty vertex set.  
The signed graph with two vertices which are connected by $t$ positive and $t$
negative edges is a subdivision of $-tC_1$. All other 
connected critical frustrated signed graphs have at least three vertices. 
As one might expect, it is easy to decide whether a critical signed graph
is a subdivision of another one.

\begin{theorem} \label{Non_decomp_2_neighbors} 
	Let $(G,\Sigma)$ be a 
	non-decomposable $k$-critical signed graph with at least three vertices. Then
	$(G,\Sigma)$ is a proper subdivision of a signed graph $(H, \Gamma)$
	if and only if $G$ has a vertex with precisely two neighbors. 
\end{theorem}

\begin{proof}	
	($\leftarrow$) By Lemma \ref{Min_signature} we can assume that 
	$\Sigma$ is a $k$-signature. 
	Let $v$ be a vertex with precisely two 
	neighbors $x$ and $y$. Let $E_{vx}$ and $E_{vy}$ be the set of
	edges between $v$ and $x$ and between $v$ and $y$, respectively. 
	Since $(G,\Sigma)$ is non-decomposable, 
	it follows by Proposition \ref{decomp_trivial} that
	$E_{vx}$ and $E_{vy}$ have a sign. Since $(G,\Sigma)$ is critical,
	we can additionally assume that $E_{vy}$ is negative. 
	
	If $|E_{vx}| < |E_{vy}|$, then $|\partial(v) \Delta \Sigma| < k$,
	a contradiction. If $|E_{vy}| < |E_{vx}|$, then there is a
	$k$-signature $\Sigma_1$ with $E_{vx} \subseteq \Sigma_1$.
	But then $|\partial(v) \Delta \Sigma_1| < k$, a contradiction.
	Thus, $|E_{vx}| = |E_{vy}| = t$.
	
	Let $(H,\Gamma)$ be the signed graph with $V(H) = V(G)-v$ and $E(H) = E(G-v) \cup E_{xy}$, where $E_{xy}$ is a set of $t$ edges between $x$ and $y$, and $\Gamma = (\Sigma \cap E(G-v)) \cup E_{xy}$.
	Now it is easy to see that $(G,\Sigma)$ is a proper subdivision of 
	$(H,\Gamma)$ (which indeed is also $k$-critical by
	Theorem \ref{Sub_Dec}).   
	The other direction of the statement is trivial.  
\end{proof}

\section{Characterizations of 1- and 2-critical signed graphs} \label{Section 1-2-critical}

We call a signed graph
\emph{irreducible} if it is not a proper subdivision of another signed graph
or it has precisely one vertex. 
 
The disjoint union of two negative circuits of 
length 1 is denoted by $-C_1 \overset{.}{\cup} -C_1$. Let
$L(1) = \{-C_1\}$ and 
$L(2) = \{-C_1 \overset{.}\cup -C_1, -2C_1, -K_4\}$.

\begin{theorem} \label{i-critical}
	For $i \in \{1,2\}$, a signed graph is $i$-critical if and
	only if it is a subdivision of an element of $L(i)$. 
\end{theorem}

As shown in Theorem \ref{Sub_Dec}, criticality is invariant under
subdividing edges. Thus, Theorem \ref{i-critical} is a direct consequence of the following two lemmas. 

\begin{lemma} An irreducible signed graph $(G,\Sigma)$ is $1$-critical if and only if $(G,\Sigma) = -C_1$.
\end{lemma}

\begin{proof}
	Clearly, $-C_1$ is $1$-critical.
	
	Let $(G,\Sigma)$ be $1$-critical. Hence, $(G,\Sigma)$ contains a negative circuit $C$. If there is an edge $e$
	which is not an edge of $C$, then $l(G-e, \Sigma - e)=1$,
	since $C$ is a subgraph of $G-e$, a contradiction. Since $(G,\Sigma)$
	is irreducible it follows that $(G,\Sigma) = -C_1$. 
\end{proof}

\begin{lemma} \label{2-critical}
	An irreducible signed graph $(G,\Sigma)$ is $2$-critical 
	if and only if 
	$(G,\Sigma) \in L(2)$.
\end{lemma}

\begin{proof} ($\leftarrow$) Clearly, the elements of $L(2)$ 
	are 2-critical.

	($\rightarrow$) We can assume that  $\Sigma$ is a 2-signature, $\Sigma = \{e_1,e_2\}$ and $e_i = x_iy_i$. If $(G,\Sigma)$ has
	less than four positive edges, the statement is trivial. 
	So we assume that $G$ has at least six edges.  
	
	Suppose that $(G,\Sigma)$ contains a multiedge $E_{xy}$. If it does not have a sign, then it contains a positive and a negative edge.
	Hence $(G,\Sigma)$ is decomposable by Proposition \ref{decomp_trivial} and it contains a subdivision of
	$-C_1 \overset{.}\cup -C_1$ or $-2C_1$. If it has a sign, then it contains
	precisely two edges. The graph $G-E_{xy}$ is 2-edge-connected,
	since for otherwise $(G,\Sigma)$ would have an edge-cut with 
	two negative edges and at most one positive edge, a contradiction to Lemma \ref{edge-cut}. Hence there are two edge-disjoint paths between $x$ and $y$ in $G-E_{xy}$.
	Thus, $(G,\Sigma)$ is a subdivision of $-2C_1$.
	
	We assume that $G$ is simple in the following. 
	If $G$ contains a divalent vertex, then, by Theorem \ref{Non_decomp_2_neighbors}, it is a subdivision of 
	a $2$-critical signed graph. So we can assume that $d_G(v) \geq 3$ for every vertex $v$.  
	
	If $G$ has a 2-edge-cut $E_2$, then $G-E_2$ has (precisely) 
	two components $H_1$, $H_2$. There is a 2-signature $\Sigma_1$
	which contains precisely one edge of $E_2$. The other
	edge of $\Sigma_1$ is in $E(H_1)$ or $E(H_2)$, say $E(H_2)$. 
	Since $H_1$ contains a vertex, $G$ is bridgeless
	and $d_G(v) \geq 3$ for every vertex $v$ there is
	a balanced circuit $C_b$ in $(G,\Sigma)$ with $E(C_b) \subset E(H_1)$.
	Since the second edge of $\Sigma_1$ is in $H_2$, it follows that there is a negative circuit $C_u$ in $H_2$ which is vertex-disjoint
	from $C_b$. Every $2$-signature $\Sigma_2$ which contains an edge of $C_b$ contains at least two edges of $C_b$. Hence, 
	$\Sigma_2 \cap E(C_u) = \emptyset$, a contradiction. Therefore,
	for every 2-signature and in particular for $\Sigma$, $G-\Sigma$
	is 2-edge connected and every equilibrated edge-cut of $(G,\Sigma)$ contains precisely four edges. 
	
	Since $G-\Sigma$ is 2-edge-connected, there are two 
	edge-disjoint paths $P_1(x_i,y_i), P_2(x_i,y_i)$ ($i \in \{1,2\}$)
	between $x_i$ and $y_i$. Every equilibrated edge-cut
	contains $e_1$ and $e_2$ and, therefore, each of the two
	positive edges
	is contained in one of $P_1(x_1,y_1), P_2(x_1,y_1)$ and
	in one of $P_1(x_2,y_2), P_2(x_2,y_2)$. Hence,
	$E(P_1(x_1,y_1)) \cup E(P_2(x_1,y_1)) = E(P_1(x_2,y_2)) \cup E(P_2(x_2,y_2))$. If $x_2, y_2 \in V(P_1(x_1,y_1))$, then
	$(G,\Sigma)$ is $-2C_1$, i.e.~$e_2$ is incident to one of $x_1, y_1$. Thus $x_2 \in V(P_1(x_1,y_1))$ if and only if 
	$y_2  \in V(P_2(x_1,y_1))$. Hence, $(G,\Sigma)$ contains
	a subdivision of $-K_4$. Thus, $(G,\Sigma)=-K_4$, 
	since $(G,\Sigma)$ is irreducible. 
\end{proof}

Remark: The proof of Lemma \ref{2-critical} is included to keep the paper self-contained. 
Lemma \ref{2-critical} also follows from Theorem 75.3 in \cite{Schrijver_book_C}, which states that every signed graph $(G,\Sigma)$ which does not contain a $-K_4$-minor has $l(G,\Sigma)$ pairwise edge-disjoint negative circuits. However, it might be that Theorem 75.3 of \cite{Schrijver_book_C}
and Lemma \ref{2-critical} are equivalent.

\section{Non-decomposable critical signed graphs} \label{Section_Non_decomp}

Consider the two signed Petersen graphs $(P,\Sigma_1)$ and $(P,\Sigma_2)$ in Figure \ref{P}. Both are 3-critical and by Proposition \ref{decomp_trivial},
both of them are non-decomposable. However, $(P,\Sigma_2)$ has the property that also every $2$-critical subgraph is 
non-decomposable. This is not true for $(P,\Sigma_1)$ since it contains
two edge-disjoint negative circuits.

For $k \geq 1$ 
let $S(k)$ be the set of $k$-critical signed graphs $(G,\Sigma)$ 
with the property that $(G,\Sigma)$ contains
a non-decomposable $m$-critical subgraph $(H,\Gamma)$ for every $m \in \{1, \dots,k\}$. 
Let $S = \bigcup_{i=1}^{\infty} S(i)$. 
Analogously, let $S^*(k)$ be the set of $k$-critical signed graphs $(G,\Sigma)$ with the property
that every $m$-critical subgraph $(H,\Gamma)$ is non-decomposable
for every $m \in \{1, \dots,k\}$.  Let $S^* = \bigcup_{i=1}^{\infty} S^*(i)$. 

\begin{proposition} \label{charact_S*}
	Let $(G,\Sigma)$ be a critical signed graph. Then $(G,\Sigma) \in S^*$
	if and only if $(G,\Sigma)$ contains no edge-disjoint negative circuits.
\end{proposition}

\begin{proof}
	If $(G,\Sigma) \not \in S^*$, then it has a $t$-critical subgraph $(H,\Gamma)$ 
	which can be decomposed into two critical subgraphs $(H_1,\Gamma_1)$ and $(H_2,\Gamma_2)$. Since $E(H_1) \cap E(H_2) = \emptyset$ and each of them contains negative circuits, it follows that $(G,\Sigma)$ contains two edge-disjoint negative circuits. 
	
	If $(G,\Sigma)$ contains two edge-disjoint negative circuits, then it contains a decomposable 2-critical subgraph. Hence,
	$(G,\Sigma) \not \in S^*$. 
\end{proof}

For $i \in \{1,2\}$ let $(H_i,\Gamma_i) $ be a signed graph. 
Let $v_iw_i \in E(H_i)-\Gamma_i$. The \emph{2-edge-sum}
$(H_1,\Gamma_1) \oplus_2 (H_2,\Gamma_2)$ 
is obtained from $(H_1 - v_1w_1,\Gamma_1)$ and $(H_2 - v_2w_2,\Gamma_2)$ by adding the positive edges $v_1v_2$ and $w_1w_2$. 

Let $u_i \in V(H_i)$ be a vertex of degree 3 with neighbors 
$x_i, y_i, z_i$ and such that all edges incident to $u_i$ are positive. 
The \emph{3-edge-sum}
$(H_1,\Gamma_1) \oplus_3 (H_2,\Gamma_2)$ 
is obtained from $(H_1-u_1,\Gamma_1)$ and $(H_2-u_2,\Gamma_2)$ by adding the positive edges $x_1x_2, y_1y_2$, and $z_1z_2$.

\begin{proposition} \label{i-sum}
	For $i \in \{2,3\}$ let $(G,\Sigma)$ be the $i$-edge-sum of an unbalanced signed graph $(H_1,\Gamma_1)$ and a 2-edge-connected balanced signed graph $(H_2,\Gamma_2)$. 
	
	\begin{enumerate}
		\item $(G,\Sigma)$ contains no edge-disjoint
	negative circuits if and only if $(H_1,\Gamma_1)$ contains no
	edge-disjoint negative circuits. 	
	\item Let $(G, \Sigma)$ be a $k$-critical signed graph. If $(G,\Sigma) = (H_1,\Gamma_1) \oplus_2 (H_2,\Gamma_2)$, then $(H_2,\Gamma_2)$
	is a balanced circuit. If $(G,\Sigma) = (H_1,\Gamma_1) \oplus_3 (H_2,\Gamma_2)$, then $(H_2,\Gamma_2)$
	is a balanced theta-graph, where a theta-graph consists of two vertices
	which are connected by three internally vertex-disjoint paths.
	\end{enumerate}
\end{proposition}

\begin{proof}
		1. For $i=3$ we will prove
	that $(G,\Sigma)$ contains two edge-disjoint
	negative circuits if and only if $(H_1,\Gamma_1)$ contains two
	edge-disjoint negative circuits. 
	The case $i=2$ can be proved analogously. 
	
	Let $C_1, C_2$ be two edge-disjoint negative circuits in $(G,\Sigma)$. None
	of them can be a subgraph of $G[H_2-v_2]$. If they are both subgraphs
	of $G[H_1-v_1]$, then the statement follows. Thus one of them, say $C_1$,  contains
	two edges of $\{x_1x_2, y_1y_2, z_1z_2 \}$, say $x_1x_2, z_1z_2$. We can assume that every
	$(x_2,z_2)$-path $P$ in $(G[H_2-v_2], \Sigma \cap E(H_2-v_2))$ is all-positive.
	Thus, there is a negative circuit $C_2'$ in  $(H_1,\Gamma_1)$ which
	is edge-disjoint from $C_1$, which is also a subgraph of $(H_1,\Gamma_1)$.
	
	If $C_1, C_2$ are edge-disjoint negative circuits in $(H_1,\Gamma_1)$,
	then at most one of them contains $v_1$. Since $(H_2,\Gamma_2)$ is
	balanced, this circuit can be extended to a negative circuit in
	$(G,\Sigma)$. Thus, $(G,\Sigma)$ contains two edge-disjoint
	negative circuits. \\
		2. Assume now that $(G, \Sigma)$ is a $k$-critical signed graph such that $(G,\Sigma) = (H_1,\Gamma_1) \oplus_i (H_2,\Gamma_2)$, $i \in \{2,3\}$. 
		Let $V_2= V(G) \cap V(H_2)$ and $V_1= V(G)-V_2$.
		Furthermore, since $(H_2, \Gamma _2)$ is balanced, we can assume the signed graph induced by $(G[V_2], \Sigma_{|_{G[V_2]}})$ being all-positive.
		If $(G,\Sigma)$ contains an equilibrated cut  
	$\partial_{(G,\Sigma)}(U)$ with more than one edge of $G[V_2]$,
	then there is an edge-cut in $(G,\Sigma)$
	which is a subset of $(E(G[V_1]) \cap \partial_{(G,\Sigma)}(U)) \cup \{v_1v_2, w_1w_2\}$ if $i=2$,
	and of $(E(G[V_1]) \cap \partial_{(G,\Sigma)}(U)) \cup \{x_1x_2, y_1y_2, z_1z_2\}$ if $i=3$ and
	which has more negative than positive edges; a contradiction to Lemma \ref{edge-cut}.  
	Thus, $G[V_2]$ is a tree with at most $i$ leaves and the statements follow.
		
\end{proof}	

\begin{proposition} \label{S_S*}
	$S(1) = S^*(1) = \{G \colon G \text{ is a subdivision of }-C_1\}$, 
	$S(2) = S^*(2) = \{G \colon G \text{ is a subdivision of } -K_4\}$ and for all $k \geq 3$: $\emptyset \not = S^*(k) \subset S(k)$. 
	In particular, $S^* \subset S$.  
\end{proposition}

\begin{proof} The statement for $k=1$ is trivial. For $k=2$, note that $-K_4$ is the only non-decomposable $2$-critical signed graph (see also Lemma \ref{2-critical}). 
	
	Clearly, $S^*(k) \subseteq S(k)$. We will show that the sets are not empty and
	that there are signed graphs in $S(k)$ which are not in $S^*(k)$, for $k \geq 3$.
	Reed \cite{Reed_Mangoes1999} proved that for every $s \geq 2$ there exists an $s$-frustrated signed graph $(G_s,\Sigma_s)$ which does not contain any edge-disjoint negative circuits. 
	By Proposition \ref{critical_subgraphs},
	$(G_s,\Sigma_s)$ contains an $m$-critical subgraph $(H,\Gamma)$ for each $m \in \{1,\dots,s\}$. The edge sets of any two negative circuits of $(H,\Gamma)$ have a non-empty intersection. 
	Thus, $(H,\Gamma) \in S^*(m)$ by Proposition \ref{charact_S*}
	and $S^*(k) \not = \emptyset$ for each $k \geq 1$.

	Consider the antibalanced odd wheel $-W_{2t+1}$ for $t \geq 2$. 
	It holds $l(-W_{2t+1}) = t+1$ and every $(t+1)$-signature 
	contains precisely one edge $e_o$ of the outer circuit $C_{2t+1}$	and 
	$t$ spokes. It is easy to see that when removing one negative
	spoke and one positive spoke which are edges of an induced triangle in $-W_{2t+1}$, then the resulting signed graph is $t$-critical and it is
	a subdivision of $-W_{2t-1}$. On the other hand, $-W_{2t+1} - e_o$ contains a $t$-critical subgraph which is a subdivision of $-tC_1$.
	Thus, $-W_{2k+1} \in S(k+1) - S^*(k+1)$ for every $k \geq 2$.
\end{proof}

Signed graphs which do not contain two edge-disjoint negative circuits
are characterized by Lu et al. \cite{Lu_Lou_CQ_2018}, where the case $k=2$ is proved in \cite{seymour1980disjoint, thomassen19802}. 
Let $\hat{G}$ be a contraction of a graph $G$ and let $x \in V(G)$. Then $\hat{x}$ denotes
the vertex of $\hat{G}$ which $x$ is contracted into. 

\begin{theorem} [\cite{Lu_Lou_CQ_2018}] \label{Charact_no_edge_disj}
	Let $(G,\Sigma)$ be a 2-connected $k$-frustrated ($k \geq 2$) signed
	graph and $\Sigma = \{x_1y_1, \dots, x_ky_k\}$ be a $k$-signature. Then the following statements are equivalent.
	\begin{enumerate}
		\item $(G,\Sigma)$ contains no edge-disjoint negative circuits.
		\item $G - \Sigma$ is contractible to a 2-connected graph containing no edge-disjoint $(\hat{x}_i,\hat{y}_i)$-path and
		$(\hat{x}_j,\hat{y}_j)$-path for any $i \not= j$.
		\item The graph $G$ can be contracted to a cubic graph $\hat{G}$ such that either 
		$\hat{G} - \{\hat{x}_1 \hat{y}_1, \dots, \hat{x}_k \hat{y}_k\}$ is a circuit $C_1$ with vertex set 
		$\{\hat{x}_1, \dots, \hat{x}_k, \hat{y}_1, \dots, \hat{y}_k\}$ or
		it can be obtained from a 2-connected plane cubic graph 
		by selecting a facial circuit $C_2$ and inserting vertices
		$\hat{x}_1, \dots, \hat{x}_k, \hat{y}_1, \dots, \hat{y}_k$
		on the edges of $C_2$ in such a way that for every 2-element set 
		$\{i,j\} \subseteq \{1, \dots,k\}$, the vertices 
		$\hat{x}_i, \hat{x}_j, \hat{y}_i, \hat{y}_j$ are around the 
		circuit $C_1$ or $C_2$ in this cyclic order. 
	\end{enumerate}	
\end{theorem}

We can now prove the main theorem which characterizes $S^*$ as a specific class of projective planar signed cubic graphs. 

\begin{theorem} \label{CS*2}
	Let $k \geq 1$ and $(G,\Sigma)$ be an irreducible $k$-critical signed graph and
	$\Sigma = \{ x_1y_1, \dots, x_ky_k\}$. Then $(G,\Sigma) \in S^*$ if and only if 
	\begin{enumerate}
		\item $(G,\Sigma) \in \{-C_1, -K_4\}$ or
		\item $(G,\Sigma)$ is obtained from a 2-connected plane cubic graph $H$ 
		by selecting a facial circuit $C_H$ and inserting vertices
		$x_1, \dots, x_k, y_1, \dots, y_k$
		on the edges of $C_H$ in such a way that for every pair 
		$\{i,j\} \subseteq \{1, \dots,k\}$, the vertices 
		$x_i, x_j, y_i, y_j$ are around the 
		circuit in this cyclic order.
	\end{enumerate}
Furthermore, for $k \geq 3$: If $(G,\Sigma) \in S^*(k)$ is irreducible,
then $G$ is a cyclically 4-edge connected projective-planar cubic graph. 
 \end{theorem}

\begin{proof} Let $(G,\Sigma) \in S^*(k)$. 
	If $k \in \{1,2\}$, then $(G,\Sigma) \in \{-C_1, -K_4\}$ by
	Proposition \ref{S_S*}.
	
	By Proposition \ref{charact_S*}, $(G,\Sigma)$ does not contain
	two edge-disjoint negative circuits. Thus, by Theorem \ref{Charact_no_edge_disj},
	$G$ can be contracted to a cubic graph $\hat{G}$ such that either 
	$\hat{G} - \{\hat{x}_1 \hat{y}_1, \dots, \hat{x}_k \hat{y}_k\}$ is a circuit $C$ with vertex set 
	$\{\hat{x}_1, \dots, \hat{x}_k, \hat{y}_1, \dots, \hat{y}_k\}$ or $\hat{G}$
	can be obtained from a 2-connected plane cubic graph $H$
	by selecting a facial circuit $C_H$ and inserting vertices
	$\hat{x}_1, \dots, \hat{x}_k, \hat{y}_1, \dots, \hat{y}_k$
	on the edges of $C_H$ in such a way that for every 2-element set 
	$\{i,j\} \subseteq \{1, \dots,k\}$, the vertices 
	$\hat{x}_i, \hat{x}_j, \hat{y}_i, \hat{y}_j$ are around the 
	circuit $C$ or $C_H$ in this cyclic order. 
	
	Let $\hat{\Sigma} = \{\hat{x}_1 \hat{y}_1, \dots, \hat{x}_k \hat{y}_k\}$ be the corresponding $k$-signature 
	on $\hat{G}$.
	We can assume $k \geq 3$ and thus, $\hat{G} - \hat{\Sigma}$ is a subdivision of a cubic graph, 
	i.e.~the second case of the above statement applies. We will show that $(G,\Sigma) = (\hat{G},\hat{\Sigma})$.
	
	Let $X = \{x_1, \dots, x_k, y_1, \dots, y_k\} \subseteq V(G)$ and 
	$\hat{X} = \{\hat{x}_1, \dots, \hat{x}_k, \hat{y}_1, \dots, \hat{y}_k\} \subseteq V(\hat{G})$.
	
	Suppose to the contrary that there is $\hat{s} \in V(\hat{G})$, which is the result of a contraction of 
	a subgraph $G[S]$ of $G$ with $s \in S$ and $|S| > 1$. Since $\hat{X} \subseteq V(\hat{G})$
	and $\hat{\Sigma} \subseteq E(\hat{G})$, it follows that $|X \cap S| \leq 1$. Furthermore,
	$\Sigma \cap E(G[S]) = \emptyset$, $|\partial_{(G,\Sigma)}(S)| = 3$, and at least two edges of
	$\partial_{(G,\Sigma)}(S)$ are positive. If all three edges are positive (i.e.~$|X \cap S|=0$), then
	the statement follows from the fact that $(G,\Sigma)$ is irreducible and Proposition \ref{i-sum}. 
	
	Now, let $|X \cap S| = 1$, say $x_1 \in S$. If there is an equilibrated edge-cut $\partial_{(G,\Sigma)}(U)$
	of $(G,\Sigma)$ that contains more than one edge of $G[S]$, then there is an edge-cut
	in $(G,\Sigma)$ which is a subset of $(E(G[V(G)-S]) \cap \partial_{(G,\Sigma)}(U)) \cup \partial_{(G,\Sigma)}(S)$
	and which has more negative than positive edges; a contradiction to Lemma \ref{edge-cut}. 
    Thus, $G[S]$ is a tree with at most two leaves. Since $(G,\Sigma)$ is irreducible it
    follows that $S=\{x_1\}$ and there is nothing to contract.  
	
	Hence, $(G,\Sigma) = (\hat{G},\hat{\Sigma})$ and $(G,\Sigma)$ 
	is obtained from a 2-connected plane cubic graph $H$ 
	by selecting a facial circuit $C_H$ and inserting vertices
	$x_1, \dots, x_k, y_1, \dots, y_k$
	on the edges of $C_H$ in such a way that for every pair 
	$\{i,j\} \subseteq \{1, \dots,k\}$, the vertices 
	$x_i, x_j, y_i, y_j$ are around the 
	circuit in this cyclic order. Clearly, $G$ has an embedding into the projective plane. 
	Furthermore, if $k \geq 3$, then the vertices $x_1,x_2,x_3,y_1,y_2,y_3$
	are the six trivalent vertices of a subdivision of a $K_{3,3}$. Hence, $G$ is not planar. 
	
	It remains to prove that $G$ is cyclically 4-edge-connected. 
	By the same arguments as above, $(G,\Sigma)$ has no non-trivial 3-edge-cut.
	Suppose to the contrary that $G$ has a 2-edge-cut $\partial_{(G,\Sigma)}(U)$. Then
	$\partial_{(G,\Sigma)}(U) \subseteq E(H)$ and $\partial_{(G,\Sigma)}(U) \cap \Sigma = \emptyset$.  
	If $\partial_{(G,\Sigma)}(U) \not \subseteq E(C_H)$, then $\partial_{(G,\Sigma)}(U)$ is an edge-cut in 
	$H$. 
	By Proposition \ref{i-sum}, one component of $G-\partial_{(G,\Sigma)}(U)$ is a path, contradicting the fact that $G$ is cubic.
	If $\partial_{(G,\Sigma)}(U) \subseteq E(C)$, then one part of $C-\partial_{(G,\Sigma)}(U)$ does not
	contain any vertex of $X$. But then, we deduce a contradiction again with Proposition \ref{i-sum}.
	Thus, $G$ is cyclically 4-edge-connected.   
\end{proof}

Figure \ref{fig:Petersen_N1} shows the signed Petersen graph $(P,\Sigma_2)$ 
embedded into the projective plane. The planar graph to start with is $K_4$.

\begin{figure}[h]
	\centering
	\includegraphics[width=0.3 \linewidth]{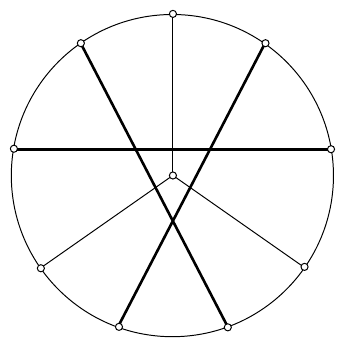}
	\caption[subdivision]{$(P,\Sigma_2)$ embedded into the projective plane; $\Sigma_2$ is indicated by the bold edges.}
	\label{fig:Petersen_N1}
\end{figure}

\section{Families of signed graphs in $S^*$} \label{Section_Family}

As we already observed, all the critical subgraphs of the Escher walls described by Reed \cite{Reed_Mangoes1999} belong to $S^*$.
Thus, $S^*(k) \not = \emptyset$ for every $k \geq 1$. However, all these graphs are
subdivisions of cubic graphs, whose structural properties are not that obvious. 
In this section we 
construct signed cubic graphs $(E_k,\Sigma_k) \in S^*(k)$ for every $k \geq 3$. Let 
${\cal W} = \{(E_k,\Sigma_k) \colon k \geq 3)\}$ be this family.
We first give the construction of the elements of ${\cal W}$ and then prove that its elements belong to $S^*$. 

\subsection*{Construction of the elements of ${\cal W}$}
 
A \emph{row} $R^i$ of length $k$ is the graph obtained from two distinct paths
$P^i=v^i_1,\dots,v^i_{2k+1}$ and $Q^i= w^i_1,\dots,w^i_{2k+1}$ with vertex set
$V(R^i)= V(P^i) \cup V(Q^i)$ and edge set $E(R^i)= E(P^i) \cup E(Q^i) \cup \{v^i_{2j+1} w^i_{2j+1} \colon j\in \{0,\dots,k\} \}$.
The circuits of length 6 in $R_i$ are called \emph{bricks}.
Furthermore, note that a row of length $k$ has exactly $k$ bricks.\\
We say that we \emph{stick} two rows $R^i$ and $R^j$ when we identify the path $Q^j$ with $P^i$ so that either $w^j_n=v^i_{n-1}$, for $n\in \{2,\dots,2k+1\}$, or $w^j_n=v^i_{n+1}$, for $n \in \{1,\dots,2k\}$.

The construction will be given for the even and the odd case separately.

If $k=2t$ is an even positive integer, then $(E_k,\Sigma_k)$ is defined as follows:

We first define an \emph{even wall} $W_e(k)$ of size $k$.
Let $(R^t, \emptyset)$ be a signed row of length $k-1$.
For $j\in \{1,\dots,t-1\}$ we stick -sequentially- two rows of length $k-j-1$, $(R^{t-j}, \emptyset)$ and $(\check{R}^{t-j}, \emptyset)$, one on the top and one on the bottom so that the first vertex of the path is identified with the second vertex of row $(R^{t-j+1}, \emptyset)$ and $(\check{R}^{t-j+1}, \emptyset)$, respectively.
If a row $R^i$ or $\check{R}^i$ has more than $4i$ vertical edges, we remove all the vertical edges except for the first $2i$ edges and the last $2i$ edges.
We relabel with $x_i$, for $i\in \{1,\dots,k\}$, the first vertex of each path, from the top to the bottom, and with $y_i$, for $i \in \{1,\dots,k\}$, the last vertex of each path, from the bottom to the top, as in Figure \ref{fig:EscherWallCritical}.
The signed graph $(E_k, \Sigma_k)$ is given by adding  
the set $\Sigma_k =\{ x_iy_i \colon i\in \{1,\dots,k\} \}$ to the wall $W_e(k)$, i.e.~$E_k = W_e(k) + \Sigma_k$.
Observe that a wall of size $k$ can be constructed from a wall of size $k-2$ by adding two bricks to each row, two more rows and possibly one or two vertical edges to some rows and then shifting the vertices $x_i, y_i$, for $i\in \{1,\dots,k\}$.

\begin{figure}[h]
	\centering
	\includegraphics[width=1\linewidth]{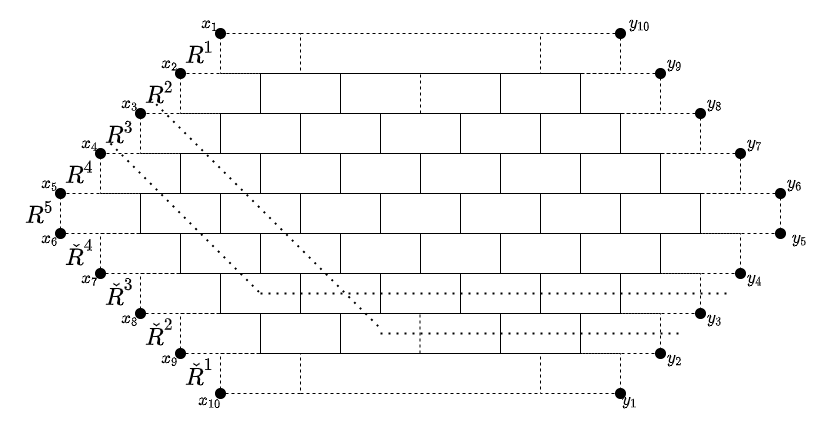}
	\caption[EscherWallCritical]{The wall of size 10 constructed from the wall of size 8 and with two edge-cuts of cardinality 10. }
	\label{fig:EscherWallCritical}
\end{figure}

If $k=2t+1$ is an odd positive integer, then $(E_k,\Sigma_k)$ is defined as follows:

We first define an \emph{odd wall} $W_o(k)$ of size $k$.
Let $(R^{t+1}, \emptyset)$ be a signed row of length $2t$.
For $j\in \{1,\dots,t\}$ we stick -sequentially- two rows $(R^{t+1-j}, \emptyset)$ and $(\check{R}^{t+1-j}, \emptyset)$ of length $2t+1-j$, the first on the top and the second on the bottom, so that the first vertex of the path of the new rows is identified with
the second vertex of row $(R^{t-j+2}, \emptyset)$ and $(\check{R}^{t-j+2}, \emptyset)$, respectively.
If a row $R^i$ or $\check{R}^i$ has more than $4i-2$ vertical edges, we remove all the vertical edges except for the first $2i-1$ edges and the last $2i-1$ edges.
We relabel with $x_i$, for $i\in \{1,\dots,k+1\}$, the first vertex of each path, from the top to the bottom, and with $y_i$, for $i\in \{1,\dots,k+1\}$, the last vertex of each path, from the bottom to the top. 
Lastly, we suppress the divalent vertices $x_{k+1}$ and $y_{k+1}$ (see Fig. \ref{fig:CriticalEscher7-9V2}).	
The signed graph $(E_k, \Sigma_k)$ is given by adding 
the set $\Sigma_k =\{ x_iy_i \colon i\in \{1,\dots,k\} \}$ to the wall $W_o(k)$, i.e.~$E_k = W_o(k) + \Sigma_k$.
In what follows, we denote by the boundary $B$ of $(E_k, \Sigma _k)$ the boundary of the outer faces of embeddings of $W_e(k)$ or of $W_o (k)$ as shown in Figures \ref{fig:EscherWallCritical} and \ref{fig:CriticalEscher7-9V2} for the cases $k \in \{9,10\}$.

\begin{figure}
	\centering
	\includegraphics[width=1\linewidth]{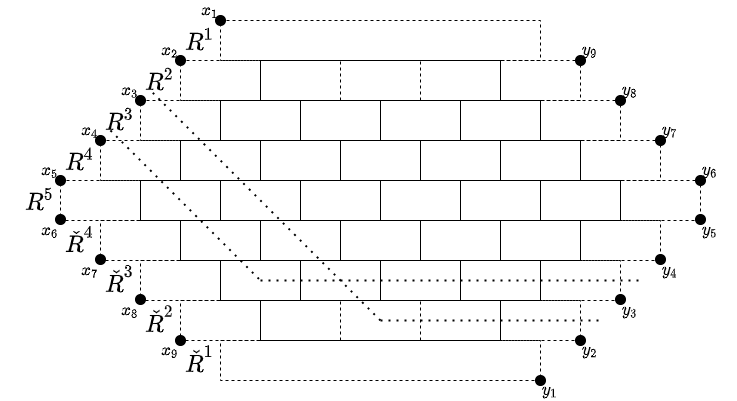}
	\caption{The wall of size 9 constructed from the wall of size 7 and with two edge-cuts of cardinality 9.}
	\label{fig:CriticalEscher7-9V2}
\end{figure}

\begin{theorem}
	For each positive integer $k \geq 3$, $(E_k, \Sigma _k) \in S^*(k)$.
\end{theorem}

\begin{proof}
	First, we prove that $(E_k, \Sigma _k)$ is $k$-frustrated, and then that it is also critical. From this together with Theorem \ref{CS*2} it follows that $(E_k, \Sigma _k) \in S^*(k)$.
		For the first step we show that for each $U \subseteq V(E_k)$, it holds $d^-(U) \leq d^+ (U)$.
		In particular, it suffices to prove this property for edge-cuts crossing $B$ exactly twice.
	To see this, observe the following.		

	Define $X= \{x_i \colon i\in \{1,\dots,k\} \}$ and $Y=\{y_i \colon i\in \{1,\dots,k\} \}$. 
	Note that, if $d^-(U) > 0$, then neither $X \cup Y \subseteq U$ nor $X \cup Y \subseteq V(G-U)$.
	In particular, we can assume that all connected components of $U$ contain at least one vertex from $X \cup Y$.
	Consider the vertices of $X \cup Y$ with the cyclic order $x_1,...,x_n,y_1,...,y_n $, that is the order they have in $B$.
	Intervals of this set belong to $U$, that is, an even number of edges of $B$ belongs to $\partial (U)$.
	Hence, $\partial (U)$ can be seen as the result of the symmetric difference of $n$ edge-cuts $\partial (U_1),..., \partial (U_n)$ such that $\partial (U_i)$, for $i\in \{1,...,n\}$, crosses $B$ exactly twice, and $\partial (U_i) \cap \partial (U_j) \cap (E(E_k)- \Sigma _k) = \emptyset$ for each $i, j \in \{1,...,n\}$, $i \neq j$.
	It follows that $d^-(U) \leq \sum _{i=1}^n d^-(U_i) \leq \sum _{i=1}^n d^+(U_i) = d^+ (U)$.
	
	Hence, we assume $\partial(U)$ being an edge-cut such that $| \partial (U) \cap E(B) |=2$.
		
	If $Y \subseteq U$ then 
	$d^-(U) = |X \cap V(G-U)|$. We can assume $X \cap V(G-U) = \{ x_i \colon i\in \{1,\dots,s\} \}$. Hence, by construction, $\partial (U)$ crosses at least $s-1$ bricks inside the wall, that is, it contains at least $s$ positive edges and $d^-(U) \leq d^+ (U)$.
	The same holds if $X \subseteq U$.
	Thus, we can assume that there exist $i', j'$ such that $x_{i'},$  $y_{j'+1} \in U$ and $x_{i'+1},$ $y_{j'} \notin U$.
	We study separately the even and the odd case.

	Let $k=2t$.
	First, assume that $i' \leq t$ and $j '\geq t $. 
	We can also assume that the edge-cut crosses the vertical edges $v^{i'}_1 w^{i'}_1$ and $v^{k-j'}_{2(k-j')+1} w^{k-j'}_{2(k-j')+1}$, that is the vertical edges of the shortest paths between $x_{i'}$ and $x_{i'+1}$ and between $y_{j'+1}$ and $y_{j'}$.
	In particular, by symmetry, we can assume that the edge-cut goes from row $R^i$, where $i=i'$, to row $R^j$, where $j=k-j'$, with $i \leq j$.
	Hence, it holds that $(\{x_s: s \in \{1,...,i\} \} \cup \{y_r : r \in \{k-j+1,...,k\} \}) \subseteq U $, and no other element of $X$ or $Y$ is in $U$. In particular, $x_s y_s \in \partial (U)$ if and only if $s \in \{1,...,i\} \cup \{k-j+1,...,k\}.$
	It implies that $d^- (U) = i+j$.
	We aim to prove that $d^+ (U) \geq i+j$.
	
	Observe that at least $j-i$ horizontal edges have to belong to the edge-cut in order to go from one row to the other, so the number of vertical edges to cross has to be at least $i+j - (j-i)= 2i$.
	But after we crossed the horizontal edges we still need to cross so many vertical edges as the number of vertical edges contained in row $R^i$,  that is either $4i$ or $i+t$.
	In both the cases it holds $d^+ (U) \geq i+j$.
	Assume now that $i'\leq t$ and $j' < t$.
	This edge-cut crosses the graph from row $R^i$, with $i=i'$, to row $\check{R}^j$, with $j=j'$.
	Again, we can assume, by symmetry, $i \leq j$.
	Since $x_s \in U$ if and only if $s \in \{1,...,i\}$, and $y_r \in U$ if and only if $ r\in \{j+1,...,2t \}$, it holds that $ x_s y_s \in \partial (U)$ if and only if $ s \in \{1,...,i\} \cup \{j+1,...,2t\} $. 
	Hence, $d^- (U) = i + 2t - j$. 
	Furthermore, the number of horizontal edges in the edge-cut is at least $ 2t-j-i$.
	We aim to prove that at least $i+2t - j -(2t - j - i)= 2i$ vertical edges belong to the edge-cut.\\
	Observe that, by crossing two horizontal edges we can also move laterally by one brick without crossing any vertical edge.
	In particular, by crossing $n$ horizontal edges, we can move laterally by $\left\lfloor \frac{n}{2} \right\rfloor $ bricks.
	That is, if we assume that row $R^i$ has $i+t$ vertical edges and since there are $j-i$ more vertical edges in $\check{R}^j$ to cross, the number of vertical edges of the edge-cut is at least 
	$i+t+(j-i)- \left\lfloor \frac{2t-j-i}{2} \right\rfloor \geq j + \frac{j+i}{2} \geq 2i$.
	Note that, after we crossed the wall vertically, by construction, we always have at least $2i$ edges to cross.
	It implies that, also if row $R^i$ has $4i < i+t$ vertical edges, by construction $2i$ of them belong to the edge-cut.
	So it holds $d^+ (U) \geq d^- (U)$. 
	
	Let $k=2t+1$.
	We first consider the case where $i'\leq t+1$ and $j'\geq t+1$.
	This edge-cut goes from row $R^i$, $i=i'$, to row $R^j$, $j=k+1-j'$.
	As in the previous case, by assuming $i \leq j$, we have that $x_s y_s \in \partial (U)$ if and only if $s \in \{1,...,i\} \cup \{k-j+2,...,k\}$ and it follows that $d^-(U) = i + j-1$.
	Since the edge-cut has to contain at least $j-i$ horizontal edges, it remains to show that there are at least $i+j - 1 - (j-i) = 2i - 1$ vertical edges to cross.
	As for the even case, we can assume that these edges are all the edges belonging to row $R^i$, so they are either $4i-2$ or $i+t \geq 2i-1$.\\
	Assume now $i'\leq t+1$ and $j'\leq t+1$,  $i' \leq j'$. 
	The edge-cut goes from row $R^i$, $i=i'$, to row $\check{R}^j$, $j=j'$.
	Repeating the same argument as in the previous cases, we have $d^-(U) = i+ 2t+1-j$, and there are at least $2t+2-i -j$ horizontal edges belonging to the edge-cut.
	We claim that
	there are at least $i+2t+1-j - (2t+2-i-j)= 2i-1$ vertical edges in the edge-cut.
	Indeed, by arguing as in the even case, since row $R^i$ has length $t+i-1$, after we moved laterally we still need to cross $t+ i + (j-i) - \left\lfloor \frac{2t+2-i-j}{2} \right\rfloor \geq j - 1 + \frac{i+j}{2} \geq 2i-1$.
	Since by construction we always have the first and the last $2i-1$ vertical edges in a row $R^i$, it implies that, also when the row has $4i-2 < t+i$ vertical edges, there are still at least $2i-1$ vertical positive edges belonging to the edge-cut.
	As a consequence, it always holds that $d^+(U) \geq d^-(U)$.
	
	It remains to prove that each edge belongs to an equilibrated edge-cut. 
For the horizontal edges not belonging to $B$, this is trivial.
Similarly, for the edges of $B$ belonging to the path from $x_1$ to $y_k$ or to the path from $x_k$ to $y_1$. \\
	For the other edges, it can be observed that the previous considerations about edge-cuts provide equality by taking $i'=j'$.
	In particular, for each row $R^i$ we can take the first $2i$ ($2i-1$ for the odd case) vertical edges, and then $k-2i$ ($k-2i+1$) horizontal edges.
	Since the horizontal edges allow the edge-cut to reach the middle of the graph, by symmetry it can be easily observed that this edge-cuts exist for all the edges.
\end{proof}

Note that this is not the only family of critical subgraphs of the Escher walls.
In the following we provide one more family ${\cal W}' = \{(E_k', \Sigma_k') \colon k = 2t+1 \text{ and } t \geq 1 \}$ for the odd case.

\begin{figure}
	\centering
	\includegraphics[width=1\textwidth]{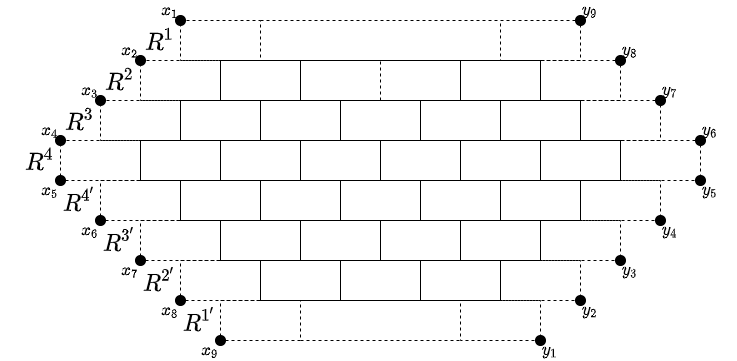}
	\caption{One more critical subgraph of the Escher wall of size 9.}
	\label{fig:CriticalEscher7-9}
\end{figure}

Let $k=2t+1$. The corresponding walls $W'(k)$ are defined as follows:
Define rows $(R^t, \emptyset)$ and $(R^{t'}, \emptyset)$ two rows of length, respectively, $2t$ and $2t-1$.
We stick a new row $(R^{t-j}, \emptyset)$ to $(R^{t-j+1}, \emptyset)$, of length $t+j$, for $j\in \{1,\dots,t-1\}$.
Similarly, we stick a new row $(R^{t-j'}, \emptyset)$ to $(R^{t-j+1'}, \emptyset)$, of length $t+j-1$, for $j\in \{1,\dots,t-1\}$.
Define $x_i$, for $i\in \{1,\dots,k\}$ the first vertex of each path, from the top to the bottom, and $y_i$, $i\in \{1,\dots,k\}$ the last  vertex of each path, from the bottom to the top, as in Figure \ref{fig:CriticalEscher7-9}.

As in the previous cases, $(E_k',\Sigma_k')$ is obtained by adding $\Sigma _k =\{x_iy_i \colon i\in \{1,\dots,k\}\}$ to $W'(k)$. That the elements of ${\cal W}'$ belong to $S^*$ can
be proved as in the previous case.
Furthermore, ${\cal W}$ and ${\cal W}'$ are different.
To see this observe, for example, that  $(E_3', \Sigma _3')$ is the signed Petersen graph.
This is not the case when we consider $(E_3, \Sigma _3)$.

\bibliography{Biblio}{}
\addcontentsline{toc}{section}{References}
\bibliographystyle{plain}

\end{document}